\documentclass[12pt]{amsart}
\usepackage{amssymb}
\usepackage{amsthm}
\usepackage{lineno}
\usepackage[pdftex, dvipsnames]{xcolor}

\usepackage[pdftex,%
letterpaper,%
includehead,%
includefoot,%
lmargin=1in,%
rmargin=1in,%
tmargin=1in,%
bmargin=1in,%
]{geometry}

\usepackage[pdftex,%
final,%
colorlinks=false,%
bookmarks=true,%
bookmarksdepth=3,%
bookmarksnumbered=true,%
bookmarksopen=true,%
bookmarksopenlevel=2,%
]{hyperref}

\allowdisplaybreaks[4]

\AtBeginDocument{%
\def\MR#1{}
}

\newtheorem{theorem}{Theorem}
\newtheorem{lemma}[theorem]{Lemma}
\newtheorem{corollary}[theorem]{Corollary}
\newtheorem{proposition}[theorem]{Proposition}

\theoremstyle{definition}

\newtheorem{definition}[theorem]{Definition}
\newtheorem{example}[]{Example}

\newtheoremstyle{teststyle} 
{10pt} 
{10pt} 
{\slshape} 
{} 
{\bfseries} 
{.} 
{ } 
{\thmname{#1}\thmnumber{ #2}\thmnote{ (#3)}} 

\theoremstyle{teststyle}
\newtheorem{test}{Test}

\newcommand{\fm}{\mathbb{F}[\mathcal{M}]}

\DeclareMathOperator{\Hom}{Hom}

\newcommand*{\barX}{\overline{X}}

\newcommand{\omk}{\Omega_K(s,t)}
\newcommand*{\co}{\colon\thinspace}
\newcommand*{\Z}{\mathbb{Z}}
\newcommand*{\Q}{\mathbb{Q}}
\newcommand*{\C}{\mathbb{C}}
\newcommand*{\F}{\mathbb{F}}
\newcommand*{\calk}{\mathcal{K}}
\newcommand*{\calm}{\mathcal{M}}
\newcommand*{\cals}{\mathcal{S}}
\newcommand*{\cs}{\mathbin{\#}}
\DeclareMathOperator{\ch}{char}
\DeclareMathOperator{\res}{Res}

\DeclareMathOperator{\GL}{GL}
\DeclareMathOperator{\HFK}{HFK}

\newcommand*{\HFKh}{\widehat{\HFK}}

\begin{document}
\title[Knot primality: knot Floer and twisted homology]{Knot primality: knot Floer homology, metacyclic representations, and twisted homology}

\author{Samantha Allen}
\address{Department of Mathematics and Computer Science, Duquesne University, Pittsburgh, PA 15282}
\email{\href{mailto:allens6@duq.edu}{allens6@duq.edu}}
\urladdr{\url{https://samantha-allen.github.io/}}

\author{Charles Livingston}
\address{Department of Mathematics, Indiana University, Bloomington, IN 47405}
\email{\href{mailto:livingst@iu.edu}{livingst@iu.edu}}
\urladdr{\url{https://chuck-livingston.github.io/math}}

\begin{abstract} 
We develop purely algebraic methods for proving that a knot is prime. Our approach uses the Heegaard Floer polynomial $\omk$ in conjunction with classical knot-theoretic methods: cyclic, dihedral, and metacyclic covering spaces. The theory of twisted homology allows us to view these approaches from a unified perspective. Collectively, the tests developed here have proved primality for  99.67\% of knots in a large family of prime knots that includes all prime knots with 15 or fewer crossings.

There are   additional ways in which our approach highlights the power of Heegaard Floer methods. For one, a single computation can prove the primality of an infinite family of knots.  We also illustrate the application of our approach to the setting of general three-manifolds by proving the primality of a knot in a nontrivial homology sphere.
\end{abstract}

\maketitle

\section{Introduction}

One of the central results in knot theory is the existence and uniqueness of prime decompositions of knots. Although the notion of primeness was formally defined by Schubert~\cite{MR31733} in 1949, the notion played a role in essentially all enumerations of knots, going back to Tait's 1885 work~\cite{tait1885knots3}.
The continuing centrality of prime decompositions in knot theory is reflected in several ways:
(1)~many recent developments specifically concern prime knots~\cite{2024arXiv240909032A, 2023arXiv230604812B, 2025arXiv250819669G, 2025arXiv250208229K, 2025arXiv250621105X};
(2)~the enumeration of prime knots is an ongoing effort, now complete through 20 crossings~\cite{burton:LIPIcs:2020:12183, MR4877260}; and
(3)~new approaches to knot-theoretic questions are routinely tested on prime knots~\cite{2025arXiv250818263C, 2025arXiv250810864C, 2025arXiv250519573G, 2025arXiv250922939M}.

Despite this central role of knot primality, the task of establishing that a given knot is prime has, until recently, been approached largely by ad hoc means.
The introduction of computational tools such as SnapPy~\cite{SnapPy} and Regina~\cite{burton:LIPIcs:2020:12183, regina}, which analyze knots through triangulations of their complements, changed the situation. This approach has been most effective for hyperbolic knots; satellite knots continued to be handled via ad hoc means,
such as in~\cite{burton:LIPIcs:2020:12183, MR4877260}.
That geometric approach culminated in the work of He, Sedgwick, and Spreer~\cite{MR4934408}, which describes a program that determines the prime decomposition of a knot.
These developments rest on the foundational results of Haken~\cite{MR141106}, which showed that decompositions of 3-manifolds, and hence of knots, may be algorithmically derived from triangulations.

In this paper we describe a purely algebraic approach to verifying that a knot is prime.
Our method combines knot Floer theory~\cite{MR2065507} with the use of metacyclic representations of knot groups to construct invariants, a technique that dates at least as far back as~\cite{MR717222,reidemeister_knotentheorie}.
The metacyclic knot invariants we use are interpreted in the context of twisted homology, permitting us to build more refined obstructions to decomposability and algorithms that are substantially faster in practice.
From a theoretical standpoint, this approach highlights a further unexpected connection between Heegaard Floer theory, classical representation methods, and twisted knot invariants. One implication of the algebraic nature of the approach is that a single computation can apply to families of knots. Another is that the approach applies immediately to knots in homology spheres.

From a computational standpoint, it yields an effective and rapid test for knot primality.

\begin{itemize}
\item \textbf{Effectiveness.}
For the full list of 15--crossing prime knots, our tests confirm primeness for 99.67\% of the entries.
Comparable effectiveness is observed for knots with much higher crossing numbers.

\item \textbf{Speed.}
At 13 crossings, the algorithm establishes primality in about 0.015 seconds. For 15-crossing knots it increases to about 0.025 seconds, which appears to be about 50 times faster than the program described in~\cite{MR4934408}. (These time comparisons take into account the difference in processor  speeds. See Section~\ref{sec-summary} for details.) For 30--crossing knots the computation requires about one second. Individual knots with up to 40 crossings remain tractable, far exceeding the range of knots considered in~\cite{MR4934408}.

\item \textbf{Satellite knots.}
The tests are particularly effective for satellite knots, for which past proofs of primality have often required detailed geometric arguments, such as in~\cite{MR4877260}.
For  the 376 satellite between 15 and 19 crossings, our program verifies primeness of  all of them in under four  seconds, whereas~\cite{MR4934408} reports a computation time of over 400 seconds (or roughly 200 seconds adjusting for processor speed).  Primeness of all 1,315 non-hyperbolic knots through 20 crossings  was verified in roughly 9 seconds.
\end{itemize}

Our work is complementary to that of~\cite{MR4934408}, which has the remarkable property of being able to certify that a knot is \emph{composite} as well as certifying primeness. On the other hand, the nature of our methods yields a program of substantially higher speed, yielding tractability for a wider class of knots. A natural long-term goal is to combine the two approaches into a single comprehensive system: our algebraic tests could serve as a rapid preliminary filter, followed by the geometric decomposition algorithms of~\cite{MR4934408}.
Such an integration has been discussed  with Alexander He, although development has not yet begun.

\smallskip

\noindent{\textbf{Outline.}} Section~\ref{sec-defhf} reviews basic background from Heegaard Floer theory and its application to develop initial primality tests. Section~\ref{section:jones} discusses an enhancement to these tests involving the Jones polynomial.  In Section~\ref{sec-branched-covers} those tools are combined with the use of cyclic branched coverings of knots to yield further tests. Metacyclic groups are studied in Section~\ref{sec-meta1} and Section~\ref{sec-meta2}. In Section~\ref{sec-twisted} twisted homology associated to metacyclic representations is discussed and Section~\ref{sec-basic} provides some basic examples.  Section~\ref{sec-mv} presents details of relevant Mayer-Vietoris arguments with twisted homology. Next, Section~\ref{sec-metatesting} describes the primality tests that arise from the use of twisted homology.

The remainder of the paper presents a mix of related topics. In Section~\ref{sec-finite} it is explained how using finite fields for coefficients rather than complex or cyclotomic fields can accelerate the computations. Section~\ref{sec-compmethods} provides some of the details of the computational methods used. In Section~\ref{sec-homologyspheres}, we discuss applications of the methods: extensions to knots in homology spheres and showing primeness of a family of knots.  Lastly, Section~\ref{sec-summary} summarizes the computational results and provides details on the computer processors used.

\smallskip

\noindent \textit{Acknowledgments.} Our joint work~\cite{2023arXiv231111089A} with Misha Temkin and Michael Wong included the observation that if the knot Floer polynomial of a knot is irreducible, then the knot is prime. That was the starting point of the research described here. Nathan Dunfield graciously advised us on the effective use of SnapPy.  Jonathan Hanselman described to us his enhancement of the Heegaard Floer program originally written by Zolt\'an Szab\'o. Morwen Thistlethwaite provided us with data from his enumeration of 20-crossing prime knots. Cameron Gordon gave us pointers to the historical background of knot primality. Hugo Zhou provided us with computations of the Heegaard Floer polynomials of knots in manifolds other than $S^3$.   Allen was supported by NSF award DMS-2532488.
\smallskip

\noindent\textit{AI-use disclosure.}  The authors used  AI language models, predominantly ChatGPT and Claude, to assist with proofreading, organization, the identification of passages that might benefit from additional explanation or verification, and to review Sage programs. The authors reviewed and revised all such output and take full responsibility for the mathematical content and final text.


\section{Knot Floer polynomial and preliminary examples}\label{sec-defhf}
The \textit{Heegaard Floer knot polynomial} is defined to be $\omk = \sum c_{i,j}s^jt^i \in \Z[t,t^{-1},s, s^{-1}]$, where $c_{i,j}= \dim \HFKh_j(K, \F_2, i)$. These homology groups were defined in~\cite{MR2065507}, where it was proved that $c_{i,j} = c_{-i, j-2i}$ for all relevant $(i,j)$. We say that a factorization of a two-variable Laurent polynomial is \textit{positive-symmetric} if the coefficients of each factor satisfy this symmetry condition and \textit{all coefficients are nonnegative}. 
A polynomial is called \textit{positive-symmetric irreducible} if it does not factor into nontrivial positive-symmetric factors.

We use the following facts about $\omk$:
\begin{enumerate}
\item $\Omega_K(-1,t) = \Delta_K(t)$, the symmetrized Alexander polynomial of $K$.\smallskip
\item $\omk = 1$ if and only if $K$ is trivial.\smallskip
\item $\Omega_K$ is multiplicative under connected sum: $\Omega_{K \cs J}(s,t) = \omk \Omega_J(s,t)$.
\end{enumerate}

\noindent{\textbf{Notation.}} We can write $\Omega_{K}(s,t)$ uniquely as $s^it^j \sum_{m = 0}^n f_m(s)t^m$, where for each $m$, $f_m(s) \in \Z[s]$ and $f_0(0) > 0$. We write this as $\Omega_{K}(s,t) \doteq \sum_{m = 0}^n f_m(s)t^m$. The symmetry and reducibility criteria are easily translated to this normalized setting.

\subsection{The basic Heegaard Floer polynomial primeness test}
The properties of $\omk$ listed above immediately lead to the primality test developed in~\cite{2023arXiv231111089A}.

\begin{test}\label{test-1} If $\omk$ is positive-symmetric irreducible, then $K$ is prime.
\end{test}
\noindent As described in~\cite{2023arXiv231111089A}, this test alone proves the primality of about 80\% of large random samples of knots.

\subsection{Three simple examples}
The following three examples illustrate the core ideas behind the tests we will develop. For these examples, we use branched rather than unbranched covers and omit twisted homology to simplify the exposition. In the remainder of the paper, we switch to unbranched covers and twisted homology, which yield stronger results and faster algorithms.

Each of these three knots is a 2-bridge knot and thus is prime by work of Schubert~\cite{MR82104}. We use them here to keep this initial exposition short.

\begin{example}
The trefoil knot, $K = T(2,3)$, has knot Floer polynomial $\Omega_K(s,t) \doteq s^2 t^2 + st + 1$. This is irreducible, and Test~\ref{test-1} applies.
\end{example}

\begin{example}
The torus knot $K = T(2,9)$ has $\Omega_K(s,t)$ with irreducible factorization $\Omega_K(s,t) \doteq (s^2t^2 + st + 1)(s^6t^6 + s^3t^3 + 1)$. If $K$ were composite, its two factors would have Alexander polynomials $(t^2 - t + 1)$ and $(t^6 - t^3 + 1)$, obtained by setting $s = -1$. This would imply that the 2-fold branched covers of the components would each have first homology $\Z_3$, obtained by setting $t = -1$. But the 2-fold cover of $T(2,9)$ has first homology $\Z_9$, while $\Z_9\ne \Z_3 \oplus \Z_3$. \end{example}

\begin{example}
The knot $K = 9_{12}$ has $\Omega_K(s,t) \doteq (s^2t^2 + 3st + 1)(2s^6t^6 + 3s^3t^3 + 2)$. If $K$ were a nontrivial connected sum $K_1 \cs K_2$, the first homology of the 2--fold branched covers would be $\Z_5$ and $\Z_7$ (which does not contradict the fact that $9_{12}$ has 2--fold branched cover $M$ with first homology $\Z_{35}$). The connected 35-fold covering space of $M$, $\widetilde{M}$,  would split along 35 2-spheres into 7 copies of the punctured 5-fold branched cover of $K_1$ and 5 copies of the punctured $7$-fold branched cover of $K_2$.    An elementary Mayer--Vietoris argument based on this  hypothesized connected sum decomposition of $\widetilde{M}$ shows that the rank of the rational first homology of this cover would be at least 24. But for $9_{12}$, $M$ is the lens space $L(35,13)$, which has 35-fold cover $\widetilde{M} = S^3$.\end{example}


\section{A Jones polynomial--Heegaard Floer polynomial primeness test} \label{section:jones}
We now describe an enhancement of Test~\ref{test-1} that was described in~\cite{2023arXiv231111089A}.
It is known that some knots are uniquely determined by their Heegaard Floer knot polynomials (and some are not~\cite{MR3782416}). Knots that are determined by their polynomials include $\{3_1, 4_1, 5_1, 5_2\}$;
see~\cite{2023arXiv231111089A, MR4760447, fsivek, MR2450204}
for details and more examples. As observed in~\cite{2023arXiv231111089A}, if $K = K_1 \cs K_2$ and the Heegaard Floer knot polynomial of one of the summands is also the polynomial of one of the knots in the full list, then one of $K_1$ or $K_2$ must be a knot $J$ in the list. This in turn would imply that the Jones polynomial of $J$ divides that of $K$. The same is true for any polynomial-valued knot invariant that behaves well under connected sum.

\begin{test}\label{test3}
If $K = K_1 \cs K_2$ is a nontrivial factorization and $\Omega_{K_i}(s,t) = \Omega_L(s,t)$ for any knot $L$ that is determined by its Heegaard Floer knot polynomial, then the Jones, Homfly, Kauffman, and properly normalized Khovanov polynomial of $L$ would divide the corresponding polynomial of $K$.
\end{test}

In our programs we have restricted to the use of Jones and Homfly tests, since the (twisted) homology computations are faster and are as effective   tools as the more sophisticated polynomial tests. As described in~\cite{2023arXiv231111089A} (restricting to only the Jones polynomial), the effectiveness of the combination of Tests~\ref{test-1} and~\ref{test3} has been about 87\%.


\section{Branched cyclic covers}\label{sec-branched-covers}

Before turning to metacyclic representations and twisted homology, we introduce higher-fold branched covers and their use along with $\omk$ to build some basic primality tests.

We will use $X(K)$ to denote the exterior of $K \subset S^3$, a compact 3-manifold with boundary a two-torus; $X_p(K)$ will denote the $p$-fold cyclic cover of $X(K)$ and $\barX_p(K)$ will denote the $p$-fold cyclic branched cover. The focus of this section is on the first homology groups, $H_1(\barX_p(K), \Z)$. These can be computed using standard classical knot theory techniques, as described in standard references~\cite{MR3156509, MR0515288}.
In particular, there is Fox's theorem~\cite{MR95876}, which we now state. The fact that prime-power covers are rational homology spheres was first noted by Fox in~\cite{MR0140099}.

\begin{proposition}
For a positive integer $p$ and knot $K$, $\big| H_1(\barX_p(K))\big| = \big| \Delta_K(t) \big|_p$, where
\[ \big| f(t) \big|_p = \big| \prod_{i = 0}^{p-1} f(\zeta_p^i)\big| \in \Z. \]
In this formula, $\zeta_p$ is a primitive $p$-th root of unity in $\C$.
If the product equals 0, then the order is infinite. If $p$ is a prime power, then the product does not equal 0.
\end{proposition}

A second basic fact about the homology of branched covers is a theorem of Plans; see~\cite{MR267567, MR0056923, MR0515288}.

\begin{theorem}
For $n$ odd, the torsion subgroup of $H_1(\barX_n(K), \Z)$ is of the form $G \oplus G$ for some finite abelian group $G$.
\end{theorem}

We have the following primality test.

\begin{test} \label{test2}
Suppose that $K$ is a nontrivial connected sum. Then there is a nontrivial positive-symmetric factorization $\omk = \Omega_1(s,t)\Omega_2(s,t)$ such that for all primes $p > 0$,
the group $H_1(\barX_p(K))$ splits as a direct sum of groups of orders $\big| \Omega_1(-1, t) \big|_p$ and $\big| \Omega_2(-1, t) \big|_p$.
\end{test}

\begin{example}
Applying this test often breaks into cases. Here is a simple example.
For the knot $K = 11_6$, we have $H_1(\barX_2(K)) \cong \Z_{27} \oplus \Z_{5}$ and $H_1(\barX_3(K)) \cong (\Z_{8} \oplus \Z_{7})^2$. The symmetric irreducible factorization of the Heegaard Floer polynomial is
\[\omk \doteq (s^2 t^2 + s t + 1)(s^2 t^2 + 3 s t + 1)(2 s^2t^2 + 5 s t + 2).\]

There are three nontrivial factorizations of $\omk$ into two factors. We consider each factorization and compare the implied information about the homology groups of the covers to the actual groups.
(1)~If one of the two factors is $s^2t^2 + st + 1$, then $H_1(\barX_2(K))$ would have a summand of order~3; there is no such summand.
(2)~If one of the two factors is $s^2t^2 + 3st + 1$, then $H_1(\barX_3(K))$ would have a summand $G \oplus G$ where $G$ is of order~4; again, no such summand exists.
(3)~If one of the two factors is $2s^2t^2 + 5st + 2$, then $H_1(\barX_2(K))$ would have a summand of order~9; this too does not occur.
\end{example}


\section{Metacyclic groups and their representations}\label{sec-meta1}

In early knot theory research, metacyclic groups were used as tools in distinguishing knots and identifying their symmetries. In this section we describe these groups and their canonical $p$-dimensional representations. The next section describes the theory of representations of knot groups into metacyclic groups.

\subsection{Metacyclic groups}
Metacyclic groups are defined by:
\[\calm(d, p, a) = \langle r, t \ \big| \ r^d = 1,\ t^p = 1,\ trt^{-1} = r^a \rangle,\]
where $a^p \equiv 1 \pmod d$.  The presentation depends only on the value of $a \pmod d$.  We will always work in a setting in which $p$ is prime, $d$ is relatively prime to $p$, and $\gcd(a-1,d) =1$.  In the case that $d = 1$ we have $\calm(1,p,1) \cong \Z_p$.  Theorem~\ref{thm-meta2} motivates the condition that $\gcd(a-1,d) = 1$.  Conjugation by $t$ induces an automorphism of the cyclic group $\langle r \rangle$ of order dividing $p$, either 1 or $p$.  When $d>1$, these assumptions imply that the class of $a$ in $(\Z/d)^\times$ has order exactly $p$.  Henceforth we will assume that all metacyclic groups satisfy these additional  constraints.  In any other case we will refer to a \textit{general metacyclic group}.

There are two special cases.
\begin{itemize}
\item In the case $d = 1$, we have $\calm(1,p,1) \cong \Z_p$.  
\item In the case $p = 2$, we have $\calm(d,2,-1)$, which is the dihedral group with $2d$ elements,   \[D_{2d} = \langle r ,t\ \big|\ r^d = 1,\ t^2 = 1,\ trt^{-1} = r^{-1}\rangle.\]
\end{itemize}

\subsection{Normal generators and metacyclic quotients}  Knot groups are normally generated by a meridian to the knot.  We will be focused on surjective representations of knot groups to metacyclic groups that send the meridian to  $t$.  Thus, the following simple observation will be useful.  

\begin{theorem} \label{thm-meta2} The general group $\calm(d,p,a)$ is normally generated by $t$ if and 
only if $\gcd(a-1,d) = 1$.  In that case, the conjugacy class of $t$ consists of elements of the form $r^it$ for $0 \le i <d$.
\end{theorem}
\begin{proof}  The quotient of $\calm(d,p,a)$ by the normal closure of the subgroup $\langle t \rangle$ is presented by $\langle r \ \big| \ r^d = 1,  r^{a-1} =1\rangle \cong  \langle r \ \big| \ r^{\gcd(d, a-1) } = 1\rangle$.   Similarly, conjugates of $t^m$  are of the form $r^{-k}t^m r^k = r^{k(a^m -1)}t^m$ for some $k$.  If $\gcd(a-1,d)=1$, then $k(a-1)$ runs over all residue classes modulo
$d$ as $k$ varies. Thus every element $r^it$ is conjugate to $t$.
\end{proof}

We will use the following  elementary facts about metacyclic groups.
 \begin{theorem} If $d= d_1d_2$ with $d_1$ and $d_2$ relatively prime, then the subgroup $\langle r^{d_1}, t\rangle$ is  isomorphic to $\calm(d_2,p,a)$. The subgroup $\langle r^{d_1} \rangle$ is normal;  the quotient of the metacyclic group by this subgroup is isomorphic to $\calm(d_1,p,a)$.
 \end{theorem}\begin{proof} The subgroup $\langle r^{d_1},t\rangle$ has generators satisfying \[ (r^{d_1})^{d_2}=1,\qquad t^p=1,\qquad tr^{d_1}t^{-1}=(r^{d_1})^a, \] and the normal form shows that it is isomorphic to $\calm(d_2,p,a)$. The subgroup $\langle r^{d_1}\rangle$ is invariant under conjugation by $r$ and by $t$, since \[ t r^{d_1} t^{-1}=r^{ad_1}=(r^{d_1})^a. \] Thus it is normal. The quotient imposes the relation $r^{d_1}=1$, giving $\calm(d_1,p,a)$, with $a$ viewed modulo $d_1$. 
 \end{proof}

\subsection{Representations of metacyclic groups}
Let $\F$ be a field containing $d$ distinct $d$-th roots of unity.
Examples include cyclotomic fields $\Q(\zeta_d)$ and finite fields of order~$q$, $\F_q$, where $q$ is a prime satisfying $q - 1 \equiv 0 \pmod d$. Elementary exercises show that these conditions imply that $d \ne 0 \in \F$ and that the multiplicative group of $d$-th roots of unity is cyclic. We will use $\zeta_d$, or simply $\zeta$ when possible, to denote a generator of the group of roots of unity.

There is an injective homomorphism $\Psi \colon \calm(d,p,a) \to \GL(p, \F)$, the set of invertible $p \times p$ matrices with coefficients in $\F$; $\Psi$ is defined by the maps illustrated below in the case that $p = 3$.

\[
t \,\mapsto\,
\begin{pmatrix}
0 & 1 & 0 \\
0 & 0 & 1 \\
1 & 0 & 0
\end{pmatrix}
\hskip .7in
r \,\mapsto\,
\begin{pmatrix}
\zeta_d & 0 & 0 \\
0 & \zeta_d^a & 0 \\
0 & 0 & \zeta_d^{a^2}
\end{pmatrix}.
\]

The proof that $\Psi$ is well-defined is a simple check that these matrices satisfy the defining relations of $\calm(d,p,a)$.

Henceforth we will refer to this as the {\em canonical linear representation} of $\calm(d,p,a)$.


\section{Metacyclic representations of knot groups}\label{sec-meta2}
We abbreviate $\pi_K = \pi_1(S^3 \setminus K)$.
For readers who are unfamiliar with metacyclic representations of knot groups, the main point of this section is that finding such representations is equivalent to labeling the arcs of a knot diagram in a certain way, and that finding such labelings is a simple linear algebra exercise over the field $\F_d$ when $d$ is prime.
We will be primarily interested in the conjugacy class of the representation, which offers some further simplifications.
\subsection{Metacyclic representations}
A homomorphism $\chi \co \pi_K \to \calm(d,p,a)$ is called a \textit{metacyclic representation} if $\chi(m) = tr^\alpha$ for some meridian $m$ and for some $\alpha \in \Z_d$.
Any conjugate of an element $tr^i$ is also of this form.
Since the meridians of a knot are all conjugate, we have that all meridians map to elements of this form.

Given an oriented diagram of $K$ with $c$ crossings, the meridians of the arcs of the diagram can be labeled $x_i$, $0 \le i < c$, where $c$ is the crossing number of the diagram.
The Wirtinger presentation of $\pi_K$ is generated by the $x_i$, and each crossing in the diagram yields a relation of the form $x_i x_j x_i^{-1} x_k^{-1}$.
Any one of these relations is a consequence of the others.
Thus, for a fixed diagram of $K$, a Wirtinger presentation of $\pi_K$ will have $c$ generators and $c-1$ relations.
If the meridians $x_i$ map to $tr^{\alpha_i}$ under a representation $\chi$, then applying $\chi$ to the Wirtinger relation $x_i x_j x_i^{-1} x_k^{-1}$ implies
\[
(1 - a)\alpha_i + a\alpha_j - \alpha_k \equiv 0 \pmod d.
\]
Conversely, any set of values $tr^{\alpha_i}$ for which the $\alpha_i$ satisfy these linear relations in $\Z_d$ determines a $\calm(d,p,a)$ representation of $\pi_K$.

   Adding a constant to each label  of a solution gives a new solution. Under our standing
assumption $\gcd(a-1,d)=1$, adding any fixed constant corresponds to
conjugating the representation by a suitable power of $r$.  Thus in searching for solutions one can restrict to the case of   $\alpha_0 = 0 \in \Z_d$. Thus, finding all $\calm(d,p,a)$ representations of $\pi_K$ is equivalent to finding all solutions to a linear system of $c-1$ equations in $c-1$ variables, where $c$ is the crossing number of the diagram.
This is most easily solved in the case that $d$ is prime, since computer algebra systems have fast routines for performing linear algebra over fields of prime order.
Some of our results require that we work in the case that $d = d_1d_2$ is the product of two distinct primes.
This can be handled by working with each prime separately and then combining the result by using the Chinese Remainder Theorem.

\subsection{Metacyclic knot group representations and branched covers}
For a prime $d$, metacyclic representations can be identified with a subgroup of $\Hom(H_1(\barX_p(K)),\Z_d)$, as we now describe.

\begin{theorem}\label{thm-br}
For a knot $K$ with meridian $m \in \pi_K$, the set $\cals$ of representations $\chi \co \pi_K \to \calm(d,p,a)$ for which $\chi(m) = t$ is in bijective correspondence with
\[
\cals' = \big\{\, \psi \co H_1(\barX_p(K)) \to \Z_d \ \big| \ \psi(T_*x) = a\psi(x) 
\text{ for all $x$} \,\big\},
\]
where $\barX_p(K)$ is the $p$-fold branched cover of $K$, and $T_*$ is the map induced by the deck transformation.
\end{theorem}

\begin{proof}
The multiplicative subgroups $\langle t\rangle$ and $\langle r \rangle$ are isomorphic to the additive groups $\Z_p$ and $\Z_d$.  In order to be consistent with standard knot theory conventions (and the desire to treat homology groups as additive), in the following argument we will place the work in the additive context when possible.

Let $\rho \co \pi = \pi_K \to \Z_p \cong \langle t\rangle  \subset \calm(d,p,a)$ be the canonical map satisfying $\rho(m) = t$.
Any representation in $\cals$,  $\chi \co \pi_K \to \calm(d,p,a)$, restricts to give a map $\tau \co \calk \to \Z_d \cong \langle  r \rangle \subset \calm(d,p,a)$, where $\calk = \ker(\rho)$.
One sees  that $\tau$ satisfies $\tau(m\gamma m^{-1}) = a\tau(\gamma)$ via the following computation (which is now written multiplicatively):    $\chi(m\gamma m^{-1})=t r^{\tau(\gamma)}t^{-1}=r^{a\tau(\gamma)}$.

Notice that $\calk = \pi_1(X_p(K))$, where $X_p(K)$ is the $p$-fold cyclic cover of the knot exterior $X(K)$;
the branched cover $\barX_p(K)$ is obtained by attaching a solid torus to $X_p(K)$.
Since $\Z_d$ is abelian, the map $\tau$ must factor through a map $\tau' \co H_1(X_p(K)) \to \Z_d$ satisfying $\tau'(T_*x) = a\tau'(x)$.
(Conjugation by $m$ acting on $\calk$ induces the map on $H_1(X_p(K))$ given by the deck transformation.)

The preimage, $\widetilde{m}$, of the meridian of  $K$ in $X_p(K)$ can be viewed as a meridian to the branched curve in $\barX_p(K)$.
It represents an invariant class in $H_1(X_p(K))$ under the action of $T_*$.

Since $\chi(m^p) = t^p = 1 \in \calm(d,p,a)$, the lifted meridian $\widetilde{m}$ satisfies $\tau'(\widetilde{m}) = 0 \in \Z_d$.   Hence, the  map $\tau'$ must vanish on the preimage of the meridian and so  can be viewed as defined on the first homology of the branched cover, as desired.

In the reverse direction, let $\tau \co H_1(\barX_p(K)) \to \Z_d$ satisfy $\tau(T_*x) = a\tau(x)$.
Any element $\gamma \in \pi_K$ can be written uniquely as a product $m^k \alpha$, where  $k$ is the image of $\gamma$ in $H_1(X(K), \Z) \cong \Z$  and $\alpha \in \calk$. 
A metacyclic  representation $\chi$ can be defined by $\chi(m^k \alpha) = t^k r^{\tau(\overline{\alpha})}$, where $\overline{\alpha}$ is the image of $\alpha$ in $H_1(\barX_p(K))$.
This sets up the bijection stated in the theorem.
\end{proof}


\section{Twisted homology associated to metacyclic representations}\label{sec-twisted}
The first time that twisted homology was used to refine dihedral knot invariants was in Casson and Gordon's initial work on knot concordance~\cite{MR900252}. There, what could have been offered as a single invariant was split into a set of finer invariants, yielding much stronger results. The papers~\cite{MR2652542,MR1670420} present the extension to metacyclic groups and, in particular, discuss the matrix representations that appeared above.

Here we begin to summarize the needed results concerning metacyclic knot group representations, expanding on the brief summary of the previous section. More details can be found in~\cite[Chapter 14A]{MR3156509}. An extended discussion of the use of local coefficients in knot theory is contained in~\cite{MR143201}. Of special note is the paper~\cite{MR683753} in which Hartley used metacyclic covers to identify all non-reversible knots of 10 crossings; his work was phrased in terms of the homology of covering spaces rather than via local coefficients.

Let $X$ be a finite connected CW complex and let $\chi \co \pi_1(X)\, \to\, \calm(d,p,a)$ be a homomorphism. Via the canonical $\GL(p, \F)$ representation, this gives the vector space $\F^p$ the structure of a right $\Z[\pi_1(X)]$-module, where we view $\GL(p, \F)$ as acting on $\F^p$ on the right. When viewed as a module in this way, we write it as $\{\F^p\}$. The definition of the twisted homology group, written as $H_i( X, \{\F^p\})$ or $H_i(X, \chi)$, will be reviewed below.

\subsection{Review of covering space theory} Before stating the main result of this section, we need to summarize basic facts about covering spaces and set up notation.

Let $X$ be a connected CW complex with a fixed base point.  Basic covering space theory states that for every normal subgroup $\pi' \subset \pi_1(X)$ there is a connected covering space $X_{\pi'}$ with fundamental group $\pi'$ and group of deck transformations isomorphic to $\pi / \pi'$.   Our convention is that deck transformations act on the left. 

This is generalized as follows.  A group homomorphism $\phi \co \pi_1(X) \to G$ induces a map of $X$ to the Eilenberg-MacLane space $K_G$.  If we denote the universal cover of $K_G$ by $E_G$, then the cover $E_G \to K_G$ pulls back  to  give a  covering space of $X$ which we denote $X_\phi$.  The group of deck transformations is isomorphic to $G$.

An alternative view of the construction of $X_\phi$ that provides some intuition goes as follows.  By attaching 1-cells to $X$, we can embed   $X \subset X'$ where  $X'$ is a space supporting  a surjection $\phi' \co \pi_1(X') \to G$ extending $\phi$.  Then $X_\phi$ is the preimage of $X$ in the connected covering space  $X'_{\ker(\phi')}$ which has group of deck transformations isomorphic to $G$.

If $\phi$ is surjective, $X_\phi = X_\calk$, where $\calk$ is the kernel of $\phi$.  More generally, the components of $X_\phi$ correspond to cosets of the image of $\phi$ and the connected components of $X_\phi$ are homeomorphic to $X_\calk$.

\subsection{Twisted homology and eigenspace decompositions}
We let ${X}_\chi$ denote the $pd$-fold covering space of $X$ corresponding to the   representation $\chi  \co \pi_1(X)\to\calm(d,p,a)$. Its group of deck transformations is isomorphic to $\calm(d,p, a)$.  The main result of this section is the following.

\begin{theorem} \label{thm-local} The groups $H_i( X, \chi)$ and $ H_i( {X}_\chi, \F)_{\zeta_d}$ are isomorphic, where $H_i( {X}_\chi, \F)_{\zeta_d}$ is the $\zeta_d$-eigenspace of the action of the order $d$ deck transformation corresponding to $r \in \calm(d, p,a)$ acting on $H_i( X_\chi, \F)$.
\end{theorem}

\subsection{Review of twisted homology groups} We abbreviate $\pi_1(X)$ by $\pi$ and let $\widetilde{X}$ denote the universal cover of $X$. Let $M$ be any right $\F[\pi]$-module. The twisted homology $H_*(X, M)$ is defined to be the homology of the chain complex of $\F$-vector spaces
\[ \cdots  \xrightarrow{\text{id} \otimes \partial_{i+1}}  M \otimes_{\F[\pi]} C_{i}(\widetilde{X}, \F) 
\xrightarrow{\text{id} \otimes \partial_i }  M \otimes_{\F[\pi]} C_{i-1}(\widetilde{X}, \F)
 \xrightarrow{\text{id} \otimes \partial_{i-1} } \cdots . \]

In the case that $M$ has the structure of an $\F[G]$-module for a group $G$ and the $\F[\pi]$-module arises via a homomorphism $\chi \co \pi\, \to\, G$, the homology of this complex is the same as that of the following complex:
\[ \cdots \xrightarrow{\text{id} \otimes \partial_{i+1}}
 M \otimes_{\F[G]} C_{i}({X}_\chi, \F) 
 \xrightarrow{\text{id} \otimes \partial_i }
  M \otimes_{\F[G]} C_{i-1}({X}_\chi, \F)  
  \xrightarrow{\text{id} \otimes \partial_{i-1}} \cdots , \]
where $X_\chi$ denotes the regular $G$-cover associated to $\chi$ as above. (As described above, if $\chi$ is surjective, then $X_\chi$ is connected with fundamental group isomorphic to $\ker \chi$. Otherwise, the number of components of $X_\chi$ is the index of the image of $\chi$. In any case, the group of deck transformations is isomorphic to $G$.)

In all that follows, we have a representation $\chi \co \pi_1(X) \to \calm(d,p,a)$.  Composing  $\chi$ with the canonical representation of $\calm(d,p,a) $ to $ \GL(p, \F)$ gives $\F^p$ the structure of a $\pi_1(X)$-module.  The twisted homology in this  case  is denoted either by $H_*(X, \chi)$ or $H_*(X, \{\F^p\})$, with the $\chi$ being implicit in the second notation.

\begin{definition}  The twisted Betti numbers are defined by $\beta_i(X , \chi) = \dim_{\F}H_i(X, \chi)$.
\end{definition}

\subsection{%
    \texorpdfstring
     {Splitting the homology when \(M = \fm\)} %
     {Splitting the homology when M = F[M]}%
 }
 To simplify notation, we will denote the $\pi_1(X)$-module $\F[\calm(d,p,a)]$ by $\F[\calm]$.  We can view $\fm$ as a right $\fm$-module via right multiplication. In the case of interest we have $\chi \co \pi \, \to\, \calm(d,p,a)$, which gives $M = \F^p$ the structure of a right $\F[\pi]$-module. 

The twisted homology $H_*(X, \fm)$ is that of the sequence
\[ \cdots \xrightarrow{\text{id} \otimes \partial_{i+1} } \fm \otimes_{\fm} C_{i}(X_\chi, \F)
\xrightarrow{  \text{id} \otimes \partial_i } \fm\otimes_{\fm} C_{i-1}(X_\chi, \F)
\xrightarrow{ \text{id} \otimes \partial_{i-1} }\cdots . \]
Observe that this is isomorphic as an $\F$-chain complex to that of $X_{\chi}$.
The element $r \in \calm(d,p,a)$ generates a normal subgroup isomorphic to $\Z_d$. Left multiplication by $r$ defines a linear transformation $R$ of $\fm$ and $\fm$ splits as the direct sum of eigenspaces of $R$. We denote the $\zeta^k$-eigenspace by $E_{\zeta^k}$. A basis for $E_{\zeta^k}$ is given by the set
\[
\left\{ \left( \sum_{j = 0}^{d-1} \zeta^{-jk} r^j \right) t^l \right\}_{l = 0,\ldots, p-1} .
\]
A proof that this is a basis will be given in Section~\ref{sec-basis}. Notice that since the left action of $R$ and the right action of $\fm$ on $\fm$ commute, the splitting into eigenspaces also splits $\fm$ as a direct sum of right $\fm$-modules. 

Given this splitting, we have that the twisted homology $H_*(X, E_\zeta)$ is given as the homology of the sequence
\[ \cdots \xrightarrow{\text{id} \otimes \partial_{i+1} }
 E_\zeta \otimes_{\fm} C_{i}(X_\chi, \F) 
 \xrightarrow{\text{id} \otimes \partial_i }
  E_\zeta\otimes_{\fm} C_{i-1}(X_\chi, \F)
  \xrightarrow{ \text{id} \otimes \partial_{i-1}} \cdots . \]

\subsection{Completion of the proof of Theorem~\ref{thm-local}}\label{sec-basis} To complete the proof of Theorem~\ref{thm-local} there are two steps: (1) identifying $E_\zeta$ with the module defined by the matrix representation presented at the beginning of this section and (2) identifying the homology of the chain complex above as giving the $\zeta$-eigenspace of $H_*(X_\chi, \F)$ under the action of $R$.

\smallskip

\noindent{\textbf{Step 1.}} To simplify notation, we consider the case of $p = 3$ and $d = 7$, so that $E_\zeta$ has basis consisting of the vectors:

\[
v_0 = 1 + \zeta^{-1} r + \zeta^{-2} r^{2} + \zeta^{-3} r^{3} + \zeta^{-4} r^4  + \zeta^{-5} r^{5} + \zeta^{-6} r^6.
\]

\[
v_1 = (1 + \zeta^{-1} r+ \zeta^{-2} r^{2} + \zeta^{-3} r^{3} + \zeta^{-4} r^4+ \zeta^{-5} r^{5} + \zeta^{-6} r^6)\, t.
\]

\[
v_2 = (1 + \zeta^{-1} r+ \zeta^{-2} r^{2} + \zeta^{-3} r^{3} + \zeta^{-4} r^4 + \zeta^{-5} r^{5} + \zeta^{-6} r^6)\, t^2.
\]

\smallskip

\noindent We now observe that these do form a basis. It is evident that multiplication on the left by $r$ multiplies each $v_i$ by $\zeta$. In particular, these vectors are contained in the eigenspace $E_\zeta$. They are clearly linearly independent, given the different exponents of $t$ in each element. A dimension-counting argument, taking into account the similarly defined bases for the other $E_{\zeta^k}$, proves that they span.

We now consider the right actions of $t$ and $r$ on the $v_k$. It is clear that right multiplication by $t$ cyclically permutes the $v_k$. We have the relation $t^k r = r^{a^k} t^k$. Thus, right multiplication of $v_k$ by $r$ yields the same element as left multiplication by $r^{a^k}$. Since we are in the $\zeta$-eigenspace for $r$, this has the same effect as multiplication by $\zeta^{a^k}$. This explains the matrix representation given earlier. (Recall, this is a right action, so the matrices are acting on the right.)

\smallskip

\noindent{\textbf{Step 2.}} This step is completed by two lemmas.

\begin{lemma} $E_\zeta \otimes_{\fm} C_{i}(X_\chi, \F) \cong C_{i}(X_\chi, \F)_\zeta$, the $\zeta$-eigenspace for $r$ acting on the left on $C_{i}(X_\chi, \F)$.
\end{lemma}

\begin{proof} Recall that $C_i(X_\chi, \F)$ is a free left ${\fm}$-module with generators corresponding to the $i$-cells of $X$. Assume there are $n$ of them. Then
\[ E_\zeta \otimes_{\fm} C_{i}(X_\chi, \F) \cong E_\zeta \otimes_{\fm} {\fm}^n \cong E_\zeta^n. \]
For later use, we note this implies that the dimension of $E_\zeta \otimes_{\fm} C_{i}(X_\chi, \F)$ equals $p n$.

The left module $C_{i}(X_\chi, \F)$ splits into eigenspaces under the left action of $r$, which we denote $C_{i}(X_\chi, \F)_{\zeta^k}$. There is a vector-space map
\[
\phi \co \bigoplus\nolimits_k C_{i}(X_\chi, \F)_{\zeta^k} \to E_\zeta \otimes_{\fm} C_{i}(X_\chi, \F)
\]
given by $\phi(x) = v_0 \otimes x$. Since the right action of $t$ cyclically permutes the set $\{v_j\}$, this is a vector-space surjection.

We claim that for $x \in C_i(X_\chi,\F)_{\zeta^k}$ one has $\phi(x) = 0$ if $k\ne 1$. To see this, we have the sequence of equalities:
\[
v_0 \otimes x = v_0 \otimes \zeta^{-k} r x = \zeta^{-k} v_0 r \otimes x = \zeta^{-k} \zeta \, v_0 \otimes x.
\]
The last equality follows from the direct calculation that $v_0r = \zeta v_0$.

\[
(1 - \zeta^{1 - k})\, v_0 \otimes x = 0.
\]
If $k\ne 1$, then this implies that $v_0 \otimes x = 0$; that is, $\phi(x) = 0$ as desired.
As a consequence, we see that the map $C_i(X_\chi, \F)_\zeta \to E_\zeta \otimes_{\fm} C_i(X_\chi, \F)$ is surjective. The dimensions are the same, so they are isomorphic. Note that $1-\zeta^{1-k}\neq 0$ for $k\neq 1$ since $\zeta$ is primitive.
\end{proof}

\begin{lemma} The homology $H_*(X_\chi, \F)_\zeta$ is isomorphic to the homology of the complex $C_*(X_\chi, \F)_\zeta$.
\end{lemma}

\begin{proof} It is clear that the inclusion $C_i(X_\chi, \F)_\zeta \to C_i(X_\chi, \F)$ induces a map on homology and the image is contained in $H_i(X_\chi, \F)_\zeta$. If a class $c \in H_i(X_\chi, \F)_\zeta$ is represented by a cycle $z$, so $c = [z]$, then $c = \zeta^{-1} R_*(c) = [\zeta^{-1}R_*z]$, where $R_*$ is induced by the deck transformation. From this we see that
\[
c = \frac{1}{d} \sum_{j = 0}^{d-1} [\zeta^{-j}R_*^j z] = \left[ \sum_{j=0}^{d-1} \frac{1}{d}\zeta^{-j}R_*^j z \right].
\]
The cycle in the brackets in the rightmost summation is in $C_i(X_\chi, \F)_{\zeta}$. It follows that $H_i(C_i(X_\chi, \F)_\zeta) \to H_i(X_\chi, \F)_\zeta$ is surjective.

The proof of injectivity is similar. If an element $[z]$ in $H_i(C_i(X_\chi, \F)_\zeta)$ represents $0$ in $H_i(X_\chi, \F)$, there is a $w \in C_{i+1}(X_\chi, \F)$ with $\partial w = z$. Now $w$ and $z$ can be averaged to show that $[z] \in H_i(C_*(X_\chi, \F)_\zeta)$ is represented by a cycle $z'$ that bounds a chain $w'$ in $C_{i+1}(X_\chi, \F)_{\zeta}$.

\end{proof}

\subsection{A comment about subgroups} For positive integers $d_1$ and $d_2$, let $d = d_1 d_2$. Temporarily we write
\[
\calm_1 = \left\langle r, t \ \b  ig|\ r^{d_1} = 1,\ t^p = 1,\ trt^{-1} = r^a \right\rangle
\qquad\text{and}\qquad
\calm = \left\langle R, T \ \big|\ R^{d} = 1,\ T^p = 1,\ TRT^{-1} = R^a \right\rangle .
\]
Notice that if the triple $(d, p, a)$ satisfies the required criteria to define a metacyclic group, then so does $(d_1, p, a)$.

The map $\phi \co \calm_1 \to \calm$ defined by $\phi(t) = T$ and $\phi(r) = R^{d_2}$ defines an injective homomorphism.

The definition of the twisted metacyclic invariants we are considering requires a choice of primitive root of unity in the field $\F$. If $\zeta_d$ is a primitive $d$-root of unity, then $\zeta_{d_1} = \zeta_{d}^{d_2}$ is a primitive $d_1$-root of unity.

\begin{lemma} With the notation established immediately above, a character $\chi \co \pi \to \calm(d_1, p, a)$ determines a character $\chi' \co \pi \to \calm(d, p, a)$. For these, we have $H_*(X, \chi) \cong H_*(X, \chi')$.
\end{lemma}
\begin{proof} This is most evident from the matrix form of the representation. Under the map $\chi$, the diagonal entries of the image of the representation of the element $r \in \calm(d_1, p, a)$ are of the form $(\zeta_d^{d_2})^{a^k}$, for $k = 0, \ldots , p-1$. Note that $\zeta_{d_1}=\zeta_d^{d_2}$.  The action of $t$ is by the same permutation matrix.

If the inclusion map $\phi$ is applied to define $\chi'$, then from $\phi(r) = R^{d_2}$ we see that for $\chi'$, the diagonal entries are of the form $(\zeta_d^{a^k})^{d_2}=(\zeta_{d_1})^{a^k},  k=0,\ldots,p-1$.  The representation matrices for $\chi$ and $\chi'$ are identical, so the twisted homology groups are the same.
\end{proof}

\begin{corollary}\label{cor-cyclic} If $\chi \co \pi_1(X) \to \calm(d,p,a)$ has cyclic image generated by $t$, then $H_*(X, \chi) \cong H_*(X_p, \F)$, where $X_p$ is the $p$-fold cyclic cover of $X$.
\end{corollary}
\begin{proof}
Let $\phi\co \pi_1(X)\to \langle t\rangle\cong \Z_p$ be the homomorphism
induced by $\chi$. Restricted to $\langle t\rangle$, the coefficient module
is naturally identified with $\F[\Z_p]$. Hence the twisted chain complex is
\[
    C_*(X,\chi)
    =
    \F[\Z_p]\otimes_{\F[\pi_1(X)]} C_*(\widetilde X;\F),
\]
which is precisely the cellular chain complex of the cover corresponding
to $\ker(\phi)$. This cover is $X_p$, so taking homology gives the result.
\end{proof}


\section{Basic examples, abelian knot group representations}\label{sec-basic}

Some basic examples and consideration of the case of abelian representations will clarify the theory and serve as foundations for later work.

Recall that $X(K)$ has the homotopy type of a CW complex with $(c-1)$ 2-cells, $c$ 1-cells and a single 0-cell. Thus, homology groups in gradings greater than $2$ are necessarily trivial; that is, $H_i(X(K))=0$ for $i>2$.

\begin{theorem} If $\chi \co \pi_K \, \to\, \calm(d,p,a)$ has   image $\langle t\rangle$, and $\ch(\F) = 0$ or $\ch(\F)$ does not divide $\big| H_1(\overline{X}_p(K), \Z)\big|$, then  $H_0(X(K), \chi)  \cong \F$, $H_1(X(K),\chi) \cong \F$, and $H_2(X(K), \chi) = 0$.
\end{theorem}

\begin{proof}  The proof of~\cite[Lemma 2]{MR900252} applies in the case of knot complements and with $p=2$ replaced with an arbitrary prime integer, to show that $H_1(\overline{X}_p(K), \Z_p) = 0$.  In particular, $H_1(\barX_p(K),\Z)$ is finite.  The lift of $K$ is $\Z$-null homologous, so there is a splitting $H_1(X_p(K)) \cong H_1(\overline{X}_p(K)) \oplus \Z$, with the $\Z$ summand generated by the meridian of the lift of $K$.   Applying the universal coefficient theorem to change  to $\F$ coefficients along with Corollary~\ref{cor-cyclic}  yields the desired result.
\end{proof}

\begin{corollary} For the space $S^1$ and representation $\chi\co \pi_1(S^1) \, \to\, \calm(d,p,a)$ sending the generator to $t$, we have $H_0(S^1 ,\chi) \cong \F$ and $H_1(S^1, \chi) \cong \F$.
\end{corollary}
\begin{proof} One can view this result as the case of $K$ unknotted, which has complement a homotopy $S^1$. Of course, given the simple CW structure of $S^1$, it can also be verified directly.
\end{proof}

There is one other basic result.

\begin{theorem}    If $\chi \co \pi_1(X(K)) \, \to\, \calm(d,p,a)$ is nonabelian, then
\begin{itemize}
\item $\beta_0(X(K), \chi ) = 0$.
\item $\beta_1(X(K), \chi) = \beta_2(X(K), \chi) \ge 1$.
\end{itemize}
\end{theorem}

\begin{proof}
We can assume $\chi$ is onto by replacing the target group with the image of $\chi$.   Given that, the $d$-fold cover of $X_p(K)$ is connected. Denote that cover by $X_\chi$. Since $H_0(X_\chi, \F) \cong \F$ and the order $d$-deck transformation acts trivially, the $\zeta$-eigenspace in $H_0(X_\chi, \F) $ is trivial, implying $\beta_0(X(K), \chi ) = 0 $. Given this, an Euler characteristic argument implies that $\beta_1(X(K), \chi) = \beta_2(X(K), \chi) $. We now prove that $\beta_2(X(K), \chi) \ge 1.$

The $dp$-fold metacyclic cover of $X(K)$ is formed as a $d$-fold cyclic cover of the $p$-fold cyclic cover of $X(K)$. The boundary of $X_\chi$ consists of $d$ tori cyclically permuted by the order $d$ deck transformation $R$. To see this, note that   on the fundamental group of the boundary torus, the longitude represents trivially since it is in the second commutator subgroup, implying that the image of the fundamental group of the boundary torus is $\langle t \rangle \in \calm(d,p,a)$, which is of index $d$.

 Denote the boundary components $\{B, RB, \ldots , R^{d-1}B\}$, where $B$ is a torus and $R$ is the deck transformation. As for all compact connected three-manifolds, the kernel of the inclusion of the second homology of the boundary into the second homology  is $1$-dimensional, generated by the sum of the classes represented by the boundary components: $\sum_{i = 0}^{d-1} R_*^i [B]$. The $\zeta$-eigenspace of the action on the boundary is not in the kernel; it is represented by $ \sum_{i = 0}^{d-1}\zeta^{-i} R_*^i [B]$.\end{proof}


\section{Mayer--Vietoris}\label{sec-mv}

Suppose $K = K_1 \cs K_2$ is a nontrivial decomposition. Then $X(K) \cong X(K_1)\cup_{A} X(K_2)$ where $A$ is an annulus. To simplify notation, we write $X(K) = X$, $X(K_1) = X_1$ and $X(K_2) = X_2$. If $\chi \co \pi_1(X) \, \to\, \calm(d,p,a)$, then we denote its restrictions to the two summands by $\chi_1$ and $\chi_2$, and its restriction to $A$ by $\chi_A$. There is the Mayer--Vietoris sequence
\begin{equation*}
\begin{split}
0 \, \to\, H_2(X_1, \chi_1 ) \oplus H_2(X_2,\chi_2) \xrightarrow{\psi_2}
H_2(X, \chi ) \xrightarrow{\delta_2} & H_1(A, \chi_A) \xrightarrow{\text{inc}_1} \\
\phantom{\hskip.3in} H_1(X_1,\chi_1) \oplus H_1(X_2,\chi_2)\xrightarrow{\psi_1}& H_1(X,\chi)\xrightarrow{\delta_1}
H_0(A, \chi_A)\xrightarrow{ } H_0(X_1, \chi_1 ) \oplus H_0(X_2,\chi_2)
\end{split}
\end{equation*}

\begin{theorem}\label{thm-abelian}
For the knot $K$, assume that the underlying field $\F$ used for the linear representation of $\calm(d,p,a)$ satisfies either $\ch(\F) = 0$ or that $\ch(\F)$ does not divide $\big| H_1(\overline{X}_p(K), \Z)\big|$.  
If $ \chi \co \pi_1(X(K)) \, \to\, \calm$ restricts so that $\chi_1$ has image $\langle t \rangle$ and $\chi_2$ is nonabelian, then
\begin{itemize}
\item $\beta_2(X, \chi ) = \beta_2(X_2, \chi_2 )$.
\item $ \beta_1(X , \chi ) = \beta_1(X_2, \chi_2 ).$
\end{itemize}
\end{theorem}
\begin{proof} We have that $H_1(A, \chi_A) \cong H_0(A, \chi_A) \cong \F$. Similarly, $H_1(X(K_1), \chi_1) \cong H_0(X(K_1), \chi_1) \cong \F$. The inclusion maps are injective: for $i = 0 $ and $1$, the map  $\operatorname{inc}_i\co   H_i(A, \chi_A) \to  H_i(X_1,\chi_1)$ is an isomorphism of the one-dimensional vector space $\F$. We thus have two sequences:

\[ 0 \, \to\, H_2(X_1, \chi_1 ) \oplus H_2(X_2,\chi_2) \xrightarrow{\psi_2} H_2(X(K), \chi) \, \to\, 0\]
and

\[ 0 \, \to\, H_1(A, \chi_A) \, \to\, H_1(X_1, \chi_1 ) \oplus H_1(X_2,\chi_2) \xrightarrow{\psi_1} H_1(X(K), \chi) \, \to\, 0.\]
For the abelian representation, $H_2(X_1, \chi_1) = 0$ and $H_1(X_1, \chi_1) \cong H_1(A, \chi_A) \cong \F$.  The two stated equalities are now apparent.\end{proof}

\begin{theorem} \label{thm-bettiadditivity} If $ \chi \co \pi_1(X(K)) \, \to\, \calm(d,p,a)$ restricts so that $\chi_1$ and $\chi_2$ are nonabelian, then
$\beta_2(X, \chi )=\beta_2(X_1, \chi_1 ) + \beta_2(X_2, \chi_2 ) +\epsilon$, where $\epsilon = 0 $ or $1$ depending on whether $\operatorname{inc}_1$ is injective or not.
\end{theorem}
\begin{proof} 
Since $H_2(A,\chi_A)=0$, exactness gives
\[
0\to H_2(X_1,\chi_1)\oplus H_2(X_2,\chi_2)
\to H_2(X,\chi)
\to \ker(\operatorname{inc}_1)\to 0.
\]
The map $\operatorname{inc}_1$ has domain $H_1(A,  \chi_A) \cong \F$ and is either injective, in which case  $H_2(X, \chi) \cong  H_2(X_1,\chi_1)\oplus H_2(X_2,\chi_2)$ and we are in the $\epsilon = 0$ case, or the map is the zero map, in which case $H_2(X, \chi) \cong  H_2(X_1,\chi_1)\oplus H_2(X_2,\chi_2) \oplus H_1(A, \chi_A) $ and we are in the $\epsilon = 1$ case. \end{proof}


\section{Metacyclic primality testing}\label{sec-metatesting}
With the Mayer--Vietoris results available, we can now describe our approach to proving knot primality. In the interest of clarity, we begin with two examples illustrating the basic cases in which the value of $d$ in $\calm(d,p,a)$ is either a product of two distinct primes or is prime. Both examples are in the dihedral setting, $p=2$. An example in the case $p = 3$ is presented in Section~\ref{sec-finite}. The general approach is briefly summarized after the second example.

\subsection{Examples}
In this section, we let $\F = \Q(\zeta_d)$. (Recall that $\zeta_d$ depends on the choice of $d$ for the metacyclic group.)

\begin{example}\label{ex-912} The knot $K = 9_{12}$ is the prime knot of smallest crossing number having a nontrivial positive-symmetric factorization.
\[ \omk = \big(2 s^2 t^2 + 3 s t + 2\big)\big( s^2 t^2 + 3 s t + 1\big). \]

If $K$ were to split as a connected sum, modulo reordering we would necessarily have that $H_1(\overline{X}_2(K_1)) \cong \Z_5$ and $H_1(\overline{X}_2(K_2) ) \cong \Z_7$. There is a unique (up to conjugation and Galois automorphism) surjective representation $\chi \co \pi_1(X(K)) \, \to\, D_{2\cdot 35}$ and also unique surjective representations $\chi_1$ and $\chi_2$ to $D_{2\cdot 5}$ and $D_{2\cdot 7}$.

Computations yield that $\beta_2(X(K), \chi) = \beta_2(X(K), \chi_1) = \beta_2(X(K), \chi_2) = 1$, yielding the failure of additivity for connected sums that is implied by Theorem~\ref{thm-bettiadditivity}.
\end{example}

\begin{example}
In the previous example, we had $H_1(\overline{X}_2(K)) \cong \Z_{r} \oplus \Z_s$, where $r$ and $s$ were distinct primes. Here we consider the case that $H_1(\overline{X}_2(K)) \cong \Z_{r} \oplus \Z_r$ for a single prime $r$. In this case, the group of dihedral representations $\calm( d, 2, -1)$ is a 2-dimensional vector space; this introduces some new structure that can be exploited.

We let $K = 10_{123}$. For this knot, one has
\[ \omk = (s^4 t^4 + 3 s^3 t^3 + 3 s^2 t^2 + 3 st + 1)^2.\]

A standard computation shows that $H_1(\overline{X}_2(K)) \cong \Z_{11} \oplus \Z_{11}$. Recall that there is a correspondence between conjugacy classes of representations of $\pi$ to $D_{2\cdot{11}}$ and nonzero elements of $H_1(\overline{X}_2(K))$. The automorphism of $\Z_{11} \oplus \Z_{11}$ given by multiplication by a nonzero element of $\Z_{11}$ corresponds to an automorphism of the dihedral group. This does not change the twisted Betti numbers, so we can view the twisted Betti numbers as functions on the projective space associated to the vector space $ \Z_{11} \oplus \Z_{11}$, usually denoted $\mathbb{P}^1(\Z_{11})$. There are 12 elements in the projective space; these are represented by $(1,0), (0,1)$ and $(1, a)$ for $1 \le a \le 10$.
If $\chi_1$ and $\chi_2$ are chosen as basis elements of $ \Z_{11} \oplus \Z_{11}$, then the projective space consists of the elements represented by $ \chi_1, \chi_2$ and $\chi_1 + i\chi_2$ for $0 < i < 11$.

Suppose now that $\chi_1$ and $\chi_2$ were chosen to be generators associated to a hypothesized splitting of $K$ as $K_1 \cs K_2$. If we let $a = \beta_2(  {X}(K_1), \chi_1) \ge 1$ and
$b= \beta_2( {X}(K_2), \chi_2) \ge 1$, then the 12 values of $\beta_2$ on the projective space would be
\[ \{ a, b\} \cup \{ a+b + \epsilon_i\}_{i = 1}^{10},\]
where for each $i$, $\epsilon_i \in \{0,1\}$.

A computation of the values yields the 12 integers: $\{1,1,1,1,1,1,1,1,1,1,2,2\}$. This implies that $a=1=b$, but this would imply that each of the remaining entries is at least 2. This contradiction shows $K$ cannot split.
\end{example}

\subsection{The general approach}
\begin{enumerate}
\item Enumerate all positive-symmetric factorizations $\omk = \Omega_1(s,t)\Omega_2(s,t)$. Apply the Jones polynomial tests for each such factorization to limit the factorizations of $\omk$ that might correspond to a nontrivial factorization $K = K_1 \cs K_2$. (In our program, computations of other polynomial invariants are not done unless the following faster tests fail to prove primality.)

\item Consider each such potential splitting $K = K_1 \cs K_2$ and use the values of $ \Omega_1(s,t)$ and $\Omega_2(s,t)$ to determine the orders of the homology groups $H_1(\overline{X}_p(K_1))$ and $H_1(\overline{X}_p(K_2))$. Check that these are consistent with the homology group $H_1(\overline{X}_p(K))$. In large surveys of knots we have found that restricting to $p=2, 3, 5,$ and $7$ offers the correct balance of effectiveness and speed before moving to more sophisticated obstructions.

\item Beginning with the prime $p=2$, find each pair of primes $d_1$ and $d_2$ dividing the orders $\big|H_1(\overline{X}_p(K_1))\big|$ and $\big|H_1(\overline{X}_p(K_2))\big|$. Assume first that $d_1\ne d_2$ and restrict attention to pairs for which $d_1$ does not divide $\big|H_1(\overline{X}_p(K_2))\big|$ and $d_2$ does not divide $\big|H_1(\overline{X}_p(K_1))\big|$.

\item  For $i=1$ and $i =2$, choose  integers $a_i$ of multiplicative order exactly $p$ in $\Z_{d_i}$. (Note, for this to exist, both  $d_1$ and $d_2$ must be congruent to  $1 \mod p$.)  By the Chinese Remainder Theorem, we can use $a_1$ and $a_2$ to  find an integer $a$ representing an element of multiplicative order exactly $p$ in $\Z_{d_1d_2}$.    Find all representations $\chi$ of $\pi_K$ into $\calm(d_1d_2, p, a)$. Compute the value of $\beta_2( {X}(K), \chi)$. The value of $\beta_2( {X}(K_1), \chi_1)$ equals that of $\beta_2( {X}(K), \chi_1)$, where $\chi_1 \co \pi_K \, \to\, \calm(d_1,p,a)$ is defined via the projection $\calm(d_1d_2,p,a) \, \to\, \calm(d_1,p,a)$. Similarly for $\beta_2({X}(K_2), \chi_2)$. Check the conditions given by Theorem~\ref{thm-bettiadditivity}.

\item Carry out a similar check in the case that $d_1 = d_2$.

\item Repeat steps~(3)-(5) for larger values of $p$.
\end{enumerate}


\section{Using finite fields versus cyclotomic fields}\label{sec-finite}
Initially, it might seem that it is natural to work in cyclotomic fields $\Q(\zeta)$. The following example illustrates the advantages of finite fields.

\begin{example} For the knot $15_{n139630}$ we can conclude from an analysis of the Heegaard Floer polynomial that if $K = K_1 \cs K_2$, then for the corresponding splitting of the $3$-fold branched cover, $\overline{X}_3(K) = \overline{X}_3(K_1) \cs \overline{X}_3(K_2)$, we have $H_1( \overline{X}_3(K_1) ) \cong \Z_7 \oplus \Z_7$ and $H_1(\overline{X}_3(K_2)) \cong \Z_{13} \oplus \Z_{13}$. One can compute all metacyclic representations and that there is a representation $\chi $ to $ \calm(91, 3, 16)$ which restricts to representations $\chi_1$ and $\chi_2$ taking values in $\calm(7, 3,  2)$ and $ \calm(13, 3, 3)$. A computation using $\Q(\zeta)$ coefficients determines that $\beta_2(X(K), \chi_1) = \beta_2(X(K), \chi_2) = \beta_2(X(K), \chi) =1$. Thus, additivity fails and $K$ is proved to be prime.

Let $q = 547$, which is prime. The multiplicative group of nonzero elements in $\F_{547} $ is a cyclic group; one can compute that $2$ is a generator. Since $q - 1= 91\cdot6$, we have that a primitive $91$st root of unity in $\F_{q} $ is  $ 2^6 = 64 \mod q$. In this case, we can compute that the Betti numbers are the same as in the cyclotomic case, $\beta_2(X(K), \chi_1)_{\F_q} = \beta_2(X(K), \chi_2)_{\F_q} = \beta_2(X(K), \chi)_{\F_q} =1$. Again, this provides a proof of primality.

Working over the cyclotomic field, the total computation time using SageMath was roughly 1.6 seconds. Working over the finite field, the computation time was reduced to about 0.01 seconds. Thus, speed is increased by a factor of more than 100.\end{example}


\section{Computational methods}\label{sec-compmethods}
Details related to algorithms for computing twisted knot invariants are contained in such papers as~\cite{MR2652542, MR1670420, MR1273784}. In this section we provide an overview.

\subsection{PD notation, Wirtinger relations, and the Fox matrix}
An $n$-crossing diagram for a knot $K$ is determined by its PD notation, an element in $(\Z^4)^n$. It is straightforward to construct from the PD notation a Wirtinger presentation of the knot group
\[ \pi_K \cong\langle x_1, \ldots , x_n \ \big| \ w_1, \ldots , w_{n-1}\rangle, \] where each $w_i $ is of the form $x_{a_i}x_{b_i} (x_{a_i})^{-1} (x_{c_i})^{-1}$. The knot complement has the homotopy type of a $2$-complex with a single $0$-cell, $n$ $1$-cells, and $(n-1)$ $2$-cells. We will denote such a complex by $X'(K)$. The Wirtinger presentation corresponds to such a decomposition.

The Fox matrix associated to the Wirtinger presentation is a $(n-1) \times n$ matrix with entries in the group ring of the free group generated by $\{x_1, \ldots , x_n\}$. Row $i$ is determined by the Wirtinger relation $w_i$ and has the following nonzero entries: in column $a_i$ there is $1 - x_{c_i}$; in column $b_i$ there is the entry $x_{a_i}$; and in column $c_i$ there is the entry $-1$. All other entries are $0$.

\subsection{%
     \texorpdfstring
       {The chain complex $C_*(\widetilde{X'}(K))$ and the twisted homology $H_2(X(K),\chi)$}%
       {The chain complex of X prime of K and twisted H2}%
   }
The CW chain complex of the universal cover of $X'(K)$ is a free $\Z[\pi_K]$-module with generators corresponding to the cells of $X'(K)$. The second boundary map $\partial_2 \co C_2(\widetilde{X'}(K)) \, \to\, C_1(\widetilde{X'}(K))$ is given by the transpose of the Fox matrix, with the $x_i$ now representing elements in $\pi_K$.

We now consider the boundary map for the chain complex with coefficients twisted by a representation $\chi \co \pi_K \, \to\, \GL(\F^p)$; this is denoted
$\partial_2^\chi\co C_2(\widetilde{X'}(K), \chi) \, \to\, C_1(\widetilde{X'}(K), \chi)$. The map is represented by a $pn \times p(n-1)$ matrix built by replacing each entry in the transposed Fox matrix with a $p\times p$ matrix. The substitution consists of applying the map from the group ring $\F[\pi_K]$ to the algebra of  $p \times p$ matrices with entries in $\F$.

The twisted homology group $H_2(X(K), \chi)$ is isomorphic to the kernel of the twisted map $\partial_2^\chi$. If one wants to compute the first homology, $H_1(X(K), \chi)$, the boundary map $\partial_1^\chi$ is easily written down. Programs such as SageMath contain routines for computing the homology of chain complexes with $\Z$ or field coefficients. As an alternative, as observed in~\cite{MR1670420} (which extended~\cite{MR1273784}), $H_1(X(K), \chi)$ can be computed directly from the second boundary map.

\subsection{%
    \texorpdfstring{The order of \(H_1(\barX_p(K))\)}%
     {The order of H1 of the p-fold branched cover of K}%
     }
Given polynomials $f(x), g(x) \in R[x]$ of degree $m$ and $n$, the \textit{resultant} is defined by
\[ \res(f,g) = (-1)^{mn}a^nb^m \prod_{i,j}( \alpha_i - \beta_j), \]
where $a$ and $b$ are the leading coefficients of $f$ and $g$, and the  indices for $\alpha_i$ and $\beta_i$  run over the full  set of roots counted with multiplicity. Factoring $f$ into linear factors reveals that this can be rewritten as
\[ \res(f,g) =  b^m \prod_i f(\beta_i) .\]

Fox's formula for the order of the homology of the $p$-fold cyclic branched cover of $K$ becomes
\[\big| H_1(\barX_p(K))\big| = \big| \res(\Delta_K , t^p -1)\big|.\]

Sylvester's matrix formula expresses the resultant in terms of a matrix; see, for instance,~\cite{MR1878556}. From this, a simple matrix manipulation provides the following result.

\begin{proposition} Let $\overline{\Delta}(t)= \sum_{i = 0}^{p-1} a_i t^i \in \Z[\Z_p] $ denote the image of the Alexander polynomial of $K$ in the quotient ring $\Z[\Z_p]$. Let $M$ denote the $p \times p$ matrix with $(i,j)$ entry $a_{\overline{i - j}}$, where $0 \le \overline{i - j} \le p -1$ and $i-j \equiv \overline{i-j} \pmod p $. Then
$\big| H_1(\barX_p(K))\big| = | \det(M) |$.
\end{proposition}

\begin{example} If $\Delta_K(t) = 1-3t+4t^2-5t^3+4t^4-3t^5+t^6$, then $\overline{\Delta}_K (t) = -3 + t +t^2$ and the order of the homology of the $3$-fold branched cover of $K$ is the absolute value of the determinant of
\[ \begin{pmatrix}
-3& 1 & 1 \\
1 & -3 & 1 \\
1 & 1 &-3 \\
\end{pmatrix} . \]
\end{example}


\section{Families of knots and knots in homology spheres}\label{sec-homologyspheres}
In this section we briefly discuss two features of the use of Heegaard Floer invariants to study knot primality. We have chosen to limit this discussion to fairly simple examples in which the relevant computations are most straightforward and alternative approaches are available.

\subsection{Families of knots} Families of knots can share the same Heegaard Floer invariants. This was highlighted in~\cite{MR3782416} where it was shown that there are infinite families of knots with identical Heegaard Floer complexes.   The first of those families includes the knot $6_1$, with
$\Omega_K(s,t) = 	2s^{-1}t^{-1}+ 5 + 2st$. (In this case, primality easily follows from the fact that the knots are all of genus 1.) Further examples demonstrating that distinct knots can have identical Heegaard Floer complexes can be found in~\cite{MR3325731}. The extent to which Heegaard Floer invariants fail to distinguish knots is made apparent by the following observation: of the 2,977 prime knots of 12 or fewer crossings, there are 852 distinct values of $\Omega_K(s,t)$.

\subsection{Knots in homology spheres}
A knot $K$ in an oriented 3-manifold $Y$ is formally a pair $(Y,K)$, where $K$ is an oriented submanifold diffeomorphic to $S^1$.  A knot is called trivial if it bounds an embedded disk in $Y$.  The connected sum of knots is defined   as in the general setting of oriented smooth manifold pairs, and a connected sum is nontrivial if both summands are nontrivial knots. Notice that with these conventions, if $K\subset B^3 \subset Y$ is a prime knot in $B^3$, then the pair $(Y, K)$ also represents a prime knot in the homology sphere $Y$.  Recall that for this operation to be well-defined, the submanifolds must be connected. Prime decompositions of knots can be defined in a manner completely analogous to that for knots in $S^3$.
 
The Heegaard Floer-polynomial obstruction to decomposability developed in this paper extends to
knots in homology 3-spheres using the Ozsv\'ath-Szab\'o~\cite{MR2065507} connected sum formula for knot 
Floer homology and the  unknot-detection theorem of Ni~\cite{MR2240902} for knots in homology spheres.
 However, computing $\Omega_K(s,t)$ when $K$ is a knot in a homology sphere is more challenging.  To illustrate the potential of such an approach, we will give a brief outline of a  simple example. 

 Let $ K= T(2,9)$, let $Y= S_1(K)$  be 1-surgery on $K$, and let $K' \subset Y$ be represented by the meridian of $K$, which is isotopic to the core of the solid torus used to perform the surgery.  Note that $K$ and $K'$ have the same complement, so both have Alexander polynomial
 \[ \Delta(t) = (t^2 - t +1)(t^6 - t^3 +1). \]
 Both knots are of genus 4 and the Alexander polynomial itself does not prove primality of either knot.
 
 Techniques for computing the homology $\HFKh(Y,K')$ were detailed in~\cite{MR3782416} and illustrated  in~\cite{MR4497224}.  Hugo Zhou, in unpublished work, provided us with a general algorithm for $L$-space knots. In the present case, one computes
 \[ s^4t^4 \Omega_{K'}(s,t) =
s^{16}t^8 + s^{15}t^7 + (s^4 + s^8 + s^9)t^6 +s^5t^5 + (2 + s^3)t^4 + s^3t^3 + (1 +s^4 +s^5)t^2 + s^9t + s^8 .\]
 This is an irreducible polynomial, proving that $(Y,K')$ is prime.


\section{Summary of Computations and Computing Equipment} \label{sec-summary}
The  knot primality testing program we wrote in Sagemath~\cite{SageMath} is available at~\cite{AllenGitHubPrime}.  A Jupyter notebook with samples is also posted there. 

As a test of the results presented here, the primality criteria were applied to all prime knots of 15 crossings or less.  Lists of these knots are available at the KnotAtlas~\cite{KnotTheory}; a  list using a different naming convention is available at the Regina website~\cite{regina}.  Knots and their mirrors are listed only once.
 We restricted to $p = 2, 3, 5$ and $7$. Going beyond that rarely gave anything stronger. The results proved primality for 99.67\% of them. See the table below for the  count of the 313,230 prime knots with 3 to 15 crossings which were proved prime by each test. Here, once a knot was proved to be prime, it is not subjected to subsequent tests.

\smallskip
\begin{center}
\begin{tabular}{|c|c|}
\hline
Test \ref{test-1} & 247,357\\
\hline
Test \ref{test3} & 43,350\\
\hline
Test \ref{test2} & 8,745\\
\hline
Metacyclic tests & 12,745\\
\hline
\end{tabular}
\end{center}
\medskip

As a second test, we considered the full list of 1,315 non-hyperbolic prime knots with at most 20 crossings enumerated by Burton~\cite{burton:LIPIcs:2020:12183} and Thistlethwaite~\cite{MR4877260}. All were proved prime by the methods presented here. We note that all but 15 of these were proved prime by Tests 1, 2, and 3, and did not  require the metacyclic tests.

In~\cite{MR4934408}, He, Sedgwick, and Spreer described the construction of a program that determines if a given knot is prime. This is stronger than what we present here, in that our tests can prove that a knot is prime, but failure of the tests does not imply that the knot is composite. Not surprisingly, our tests are much faster. For instance,~\cite{MR4934408} reports that for the 376 satellite knots with 15 to 19 crossings it took 405 seconds to prove primality. The program based on the results of this paper took under four seconds (using a somewhat faster processor).  It proved  the primality of all 1,315 non-hyperbolic prime knots with up to 20 crossings in approximately 9 seconds.

\subsection{Processors} Our computations were done using an iMac with an M1 processor, running on a single core, using SnapPy within a SageMath environment. For contrast, the work of He, Sedgwick, and Spreer~\cite{MR4934408} used an Intel Core i5-7200U processor. General benchmarks indicate that the M1 setting is roughly twice the speed of the older Intel processor. Given the exponential growth in complexity, we estimate that for knots in the range of under 20 crossings our algorithms run perhaps 50 times as fast, and that beyond something like 25 to 30 crossings, our algorithms continue to run, whereas those of~\cite{MR4934408} are no longer effective. But again, the algorithms of~\cite{MR4934408} should be viewed as complementary to ours, as they are doing more, actually proving that knots are composite.


\bibliographystyle{amsplain}
\bibliography{BibTexComplete}

\providecommand{\bysame}{\leavevmode\hbox to3em{\hrulefill}\thinspace}
\providecommand{\MR}{\relax\ifhmode\unskip\space\fi MR }
\providecommand{\MRhref}[2]{%
  \href{http://www.ams.org/mathscinet-getitem?mr=#1}{#2}
}
\providecommand{\href}[2]{#2}
\begin{thebibliography}{10}

\bibitem{AllenGitHubPrime}
Samantha Allen and Charles Livingston, \emph{Knot primality tester},
  \url{https://github.com/samantha-allen/knot-primality-tester}, Accessed:
  2026-06-29.

\bibitem{2023arXiv231111089A}
Samantha Allen, Charles Livingston, Misha Temkin, and C.~M.~Michael Wong,
  \emph{{Using knot Floer invariants to detect prime knots}}, arXiv e-prints
  (2023), arXiv:2311.11089, to appear in Pac. J. Math.

\bibitem{2024arXiv240909032A}
Taylor {Applebaum}, Sam {Blackwell}, Alex {Davies}, Thomas {Edlich}, Andr{\'a}s
  {Juh{\'a}sz}, Marc {Lackenby}, Nenad {Toma{\v{s}}ev}, and Daniel {Zheng},
  \emph{{The unknotting number, hard unknot diagrams, and reinforcement
  learning}}, arXiv e-prints (2024), arXiv:2409.09032.

\bibitem{MR4760447}
John~A. Baldwin and Steven Sivek, \emph{Characterizing slopes for {$5_2$}}, J.
  Lond. Math. Soc. (2) \textbf{109} (2024), no.~6, Paper No. e12951, 64.
  \MR{4760447}

\bibitem{fsivek}
\bysame, \emph{Floer homology and non-fibered knot detection}, Forum Math. Pi
  \textbf{13} (2025), Paper No. e1, 65. \MR{4851701}

\bibitem{KnotTheory}
Dror Bar-Natan, \emph{The {Mathematica} package {KnotTheory\textasciigrave}},
  The Knot Atlas,
  \url{http://katlas.org/wiki/The_Mathematica_Package_KnotTheory}, 2023,
  Ongoing development; manual edition of September 14, 2005.

\bibitem{2023arXiv230604812B}
Keegan {Boyle}, Nicholas {Rouse}, and Ben {Williams}, \emph{{A classification
  of symmetries of knots}}, arXiv e-prints (2023), arXiv:2306.04812.

\bibitem{MR3156509}
Gerhard Burde, Heiner Zieschang, and Michael Heusener, \emph{Knots}, extended
  ed., De Gruyter Studies in Mathematics, vol.~5, De Gruyter, Berlin, 2014.
  \MR{3156509}

\bibitem{burton:LIPIcs:2020:12183}
Benjamin~A. Burton, \emph{{The Next 350 Million Knots}}, 36th International
  Symposium on Computational Geometry (SoCG 2020) (Dagstuhl, Germany) (Sergio
  Cabello and Danny~Z. Chen, eds.), Leibniz International Proceedings in
  Informatics (LIPIcs), vol. 164, Schloss Dagstuhl--Leibniz-Zentrum f{\"u}r
  Informatik, 2020, pp.~25:1--25:17.

\bibitem{regina}
Benjamin~A. Burton, Ryan Budney, William Pettersson, et~al., \emph{Regina:
  Software for low-dimensional topology}, {\tt http://\allowbreak
  regina-normal.\allowbreak github.\allowbreak io/}, 1999--2025.

\bibitem{2025arXiv250818263C}
Jason {Cantarella}, Andrew {Rechnitzer}, Henrik {Schumacher}, and Clayton
  {Shonkwiler}, \emph{{New upper bounds for stick numbers}}, arXiv e-prints
  (2025), arXiv:2508.18263.

\bibitem{MR900252}
A.~J. Casson and C.~McA. Gordon, \emph{Cobordism of classical knots}, \`A la
  recherche de la topologie perdue, Progr. Math., vol.~62, Birkh\"auser Boston,
  Boston, MA, 1986, With an appendix by P. M. Gilmer, pp.~181--199. \MR{900252}

\bibitem{2025arXiv250810864C}
Jeevan {Chandra}, \emph{{Statistics in 3d gravity from knots and links}}, arXiv
  e-prints (2025), arXiv:2508.10864.

\bibitem{SnapPy}
Marc Culler, Nathan~M. Dunfield, Matthias Goerner, and Jeffrey~R. Weeks,
  \emph{Snap{P}y, a computer program for studying the geometry and topology of
  $3$-manifolds}, Available at \url{http://snappy.computop.org}, 2023,
  accessed: November 2, 2023.

\bibitem{MR0140099}
R.~H. Fox, \emph{A quick trip through knot theory}, Topology of 3-manifolds and
  related topics ({P}roc. {T}he {U}niv. of {G}eorgia {I}nstitute, 1961),
  Prentice-Hall, Englewood Cliffs, N.J., 1962, pp.~120--167. \MR{0140099}

\bibitem{MR95876}
Ralph~H. Fox, \emph{Free differential calculus. {III}. {S}ubgroups}, Ann. of
  Math. (2) \textbf{64} (1956), 407--419. \MR{95876}

\bibitem{MR2450204}
Paolo Ghiggini, \emph{Knot {F}loer homology detects genus-one fibred knots},
  Amer. J. Math. \textbf{130} (2008), no.~5, 1151--1169. \MR{2450204}

\bibitem{2025arXiv250819669G}
Sudipta {Ghosh}, Zhenkun {Li}, and Juanita {Pinz{\'o}n-Caicedo},
  \emph{{$SU(2)$-representations of branched covers}}, arXiv e-prints (2025),
  arXiv:2508.19669.

\bibitem{MR267567}
C.~McA. Gordon, \emph{A short proof of a theorem of {P}lans on the homology of
  the branched cyclic coverings of a knot}, Bull. Amer. Math. Soc. \textbf{77}
  (1971), 85--87. \MR{267567}

\bibitem{2025arXiv250519573G}
Vaishnavi {Gupta} and Hitesh {Raundal}, \emph{{Biorderability of knot quandles
  of knots up to eight crossings}}, arXiv e-prints (2025), arXiv:2505.19573.

\bibitem{MR141106}
Wolfgang Haken, \emph{Theorie der {N}ormalfl\"{a}chen}, Acta Math. \textbf{105}
  (1961), 245--375. \MR{141106}

\bibitem{MR683753}
Richard Hartley, \emph{Identifying noninvertible knots}, Topology \textbf{22}
  (1983), no.~2, 137--145. \MR{683753}

\bibitem{MR4934408}
Alexander He, Eric Sedgwick, and Jonathan Spreer, \emph{A practical algorithm
  for knot factorisation}, 41st {I}nternational {S}ymposium on {C}omputational
  {G}eometry, LIPIcs. Leibniz Int. Proc. Inform., vol. 332, Schloss Dagstuhl.
  Leibniz-Zent. Inform., Wadern, 2025, pp.~Art. No. 55, 15. \MR{4934408}

\bibitem{MR3782416}
Matthew Hedden and Liam Watson, \emph{On the geography and botany of knot
  {F}loer homology}, Selecta Math. (N.S.) \textbf{24} (2018), no.~2, 997--1037.
  \MR{3782416}

\bibitem{MR2652542}
Chris Herald, Paul Kirk, and Charles Livingston, \emph{Metabelian
  representations, twisted {A}lexander polynomials, knot slicing, and
  mutation}, Math. Z. \textbf{265} (2010), no.~4, 925--949. \MR{2652542}

\bibitem{MR4497224}
Jennifer Hom, Adam~Simon Levine, and Tye Lidman, \emph{Knot concordance in
  homology cobordisms}, Duke Math. J. \textbf{171} (2022), no.~15, 3089--3131.
  \MR{4497224}

\bibitem{MR1670420}
Paul Kirk and Charles Livingston, \emph{Twisted {A}lexander invariants,
  {R}eidemeister torsion, and {C}asson-{G}ordon invariants}, Topology
  \textbf{38} (1999), no.~3, 635--661. \MR{1670420}

\bibitem{2025arXiv250208229K}
Teruaki {Kitano} and Yasuharu {Nakae}, \emph{{An extended symmetric union and
  its Alexander polynomial}}, arXiv e-prints (2025), arXiv:2502.08229.

\bibitem{MR1878556}
Serge Lang, \emph{Algebra}, third ed., Graduate Texts in Mathematics, vol. 211,
  Springer-Verlag, New York, 2002. \MR{1878556}

\bibitem{2025arXiv250922939M}
Rob {McConkey} and Luke~J {Seaton}, \emph{{{Satellite operations and
  $\theta$}}}, arXiv e-prints (2025), arXiv:2509.22939.

\bibitem{MR3325731}
Allison~H. Moore and Laura Starkston, \emph{Genus-two mutant knots with the
  same dimension in knot {F}loer and {K}hovanov homologies}, Algebr. Geom.
  Topol. \textbf{15} (2015), no.~1, 43--63. \MR{3325731}

\bibitem{MR2240902}
Yi~Ni, \emph{A note on knot {F}loer homology of links}, Geom. Topol.
  \textbf{10} (2006), 695--713. \MR{2240902}

\bibitem{MR2065507}
Peter Ozsv\'{a}th and Zolt\'{a}n Szab\'{o}, \emph{Holomorphic disks and knot
  invariants}, Adv. Math. \textbf{186} (2004), no.~1, 58--116. \MR{2065507}

\bibitem{MR0056923}
Antonio Plans, \emph{Contribution to the study of the homology groups of the
  cyclic ramified coverings corresponding to a knot}, Revista Acad. Ci. Madrid
  \textbf{47} (1953), 161--193 (5 plates). \MR{0056923}

\bibitem{MR717222}
K.~Reidemeister, \emph{Knot theory}, BCS Associates, Moscow, Idaho, 1983,
  Translated from the German by Leo F. Boron, Charles O. Christenson and Bryan
  A. Smith. \MR{717222}

\bibitem{reidemeister_knotentheorie}
Kurt Reidemeister, \emph{Knotentheorie}, Ergeb. Math. Grenzgeb., Springer,
  Berlin, 1932.

\bibitem{MR0515288}
Dale Rolfsen, \emph{Knots and links}, Publish or Perish, Inc., Berkeley,
  Calif., 1976, Mathematics Lecture Series, No. 7. \MR{0515288}

\bibitem{MR31733}
Horst Schubert, \emph{Die eindeutige {Z}erlegbarkeit eines {K}notens in
  {P}rimknoten}, S.-B. Heidelberger Akad. Wiss. Math.-Nat. Kl. \textbf{1949}
  (1949), no.~3, 57--104. \MR{31733}

\bibitem{MR82104}
\bysame, \emph{Knoten mit zwei {B}r\"ucken}, Math. Z. \textbf{65} (1956),
  133--170. \MR{82104}

\bibitem{tait1885knots3}
Peter~Guthrie Tait, \emph{On knots. {III}}, Transactions of the Royal Society
  of Edinburgh \textbf{32} (1885), 446--453.

\bibitem{SageMath}
{The Sage Developers}, \emph{{SageMath, the Sage Mathematics Software System
  (Version 9.4)}}, The Sage Development Team, 2021,
  \url{https://www.sagemath.org}.

\bibitem{MR4877260}
Morwen~B. Thistlethwaite, \emph{The enumeration and classification of prime
  20-crossing knots}, Algebr. Geom. Topol. \textbf{25} (2025), no.~1, 329--344.
  \MR{4877260}

\bibitem{MR143201}
H.~F. Trotter, \emph{Homology of group systems with applications to knot
  theory}, Ann. of Math. (2) \textbf{76} (1962), 464--498. \MR{143201}

\bibitem{MR1273784}
Masaaki Wada, \emph{Twisted {A}lexander polynomial for finitely presentable
  groups}, Topology \textbf{33} (1994), no.~2, 241--256. \MR{1273784}

\bibitem{2025arXiv250621105X}
Xiaoyu {Xu}, \emph{{On regularity of profinite isomorphisms between cusped
  hyperbolic 3-manifolds and the $A$-polynomial}}, arXiv e-prints (2025),
  arXiv:2506.21105.

\end{thebibliography}

\end{document}